\documentclass{article}
\usepackage[utf8]{inputenc}
\usepackage[a4paper, total={6in, 9in}]{geometry}
\usepackage{amsmath}
\usepackage{amsthm}
\usepackage{amssymb}
\usepackage{graphicx}
\usepackage{hyperref}
\theoremstyle{definition}
\newtheorem{definition}{Definition}[section]
\newtheorem{example}{Example}[section]
\newtheorem{theorem}{Theorem}[section]
\newtheorem{corollary}{Corollary}[section]

\title{On correct structures and similarity measures of soft sets along with historic comments of Prof. D. A. Molodtsov}

\date{}

\begin{document}
\author{Santanu Acharjee$^1$ and Amlanjyoti Oza$^2$\\
$^{1,2}$Department of Mathematics\\
Gauhati University\\
Guwahati-781014, Assam, India.\\
e-mails: $^1$sacharjee326@gmail.com, $^2$amlanjyotioza21@gmail.com\\
}

\maketitle
\begin{center}
    \textbf{Abstract} \\
\end{center}
After the paper of Molodtsov [ D. Molodtsov,  Soft set theory—first results. Computers and Mathematics with Applications, 37(4-5) (1999), 19-31], soft set theory grew at a breakneck pace. Several authors have introduced various operations, relations, results, as well as other aspects in soft set theory and hybrid structures incorrectly, despite their widespread use in mathematics and allied areas. In his paper [D. A. Molodtsov,  Equivalence and correct operations for soft sets. International Robotics and Automation Journal, 4(1) (2018), 18-21], Molodtsov, the father of soft set theory, pointed out the wrong results and notions. Molodtsov [D. A. Molodtsov,  Structure of soft sets. Nechetkie Sistemy i Myagkie, 12(1) (2017), 5-18] also stated that the concept of soft set had not been fully understood and used everywhere. As a result, it seemed reasonable to revisit the quirks of those conceptions and provide a formal account of the notion of soft set equivalency. Molodtsov already explored the correct operations on soft sets. We use some notions and results of Molodtsov [Molodtsov D.A. (2017) Structure of soft sets. Nechetkie Sistemy i Myagkie, 12(1) (2017),  5-18] to  create  matrix representations as well as certain associated operations of soft sets, and to quantify the similarity between two soft sets.\\ \\
\noindent
{\bf 2020 AMS Classifications:} 03C99, 03B52,03B70, 68T27, 68T09.  
\section{Introduction}
Molodtsov \cite{1} introduced  soft set theory in 1999.  It is important to note that soft set theory has contributed to create a new area of analysis named soft rational analysis \cite{16}. But, the idea \cite{4} of soft set theory was developed in  1980. Later, various authors e.g. \cite{5,6,7,8,9,10,11,12,13} changed several  outcomes that contradicted the accuracy of the ideas and operations established in  \cite{1,2,3, 4} and others. In \cite{2}, Molodtsov mentioned, \textit{``... many authors have introduced new operations and relations for soft sets and used these structures in various areas of mathematics and in applied science. Unfortunately in some works the introduction of operations and relationships for soft sets were carried out without due regard to the specific of soft set definition"}. In section 2 of \cite{3}, the notion  of correctness of soft operations is explicitly stated. Almost all the authors of the previously stated  papers have used the incorrect notion of soft subset of a soft set \cite{2}. Molodtsov also stated in \cite{2} that the authors of \cite{6} established definitions of  complement, union, and intersection of soft sets incorrectly. Moreover,  many researchers  have utilised these incorrect notions in their works. Çağman and Enginoğlu  \cite{9} also defined a new notion of soft set and presented the matrix form of it, which takes into account a subset of the collection of parameters. This representation is erroneous because even if we are given the soft set's matrix, we cannot determine the soft set along with original set of attributes. Moreover, this notion has difference with the original notion of soft sets defined by Molodtsov \cite{1}. In this regard, we want to attract the attention of the community of soft set researchers to a comment on ResearchGate \cite{21}. On 15/11/2018, Molodtsov commented on a paper of  Al-Qudah and Hassan \cite{20}, which is available in ResearchGate till writing of this paper. He commented as follows \cite{21}:\\\\
\noindent
``{\it Dear Colleagues, I did not write such a definition of a soft set. You have added one extra condition that the set of parameters is a subset of a fixed set."}\\\\
\noindent
We provide screenshot ( figure 1) of this very important comment in the history of soft set theory for our readers. It was taken by the first author of this paper through his ResearchGate account. 
\begin{figure}[h]
    \centering
    \includegraphics[width=15cm, height=14cm]{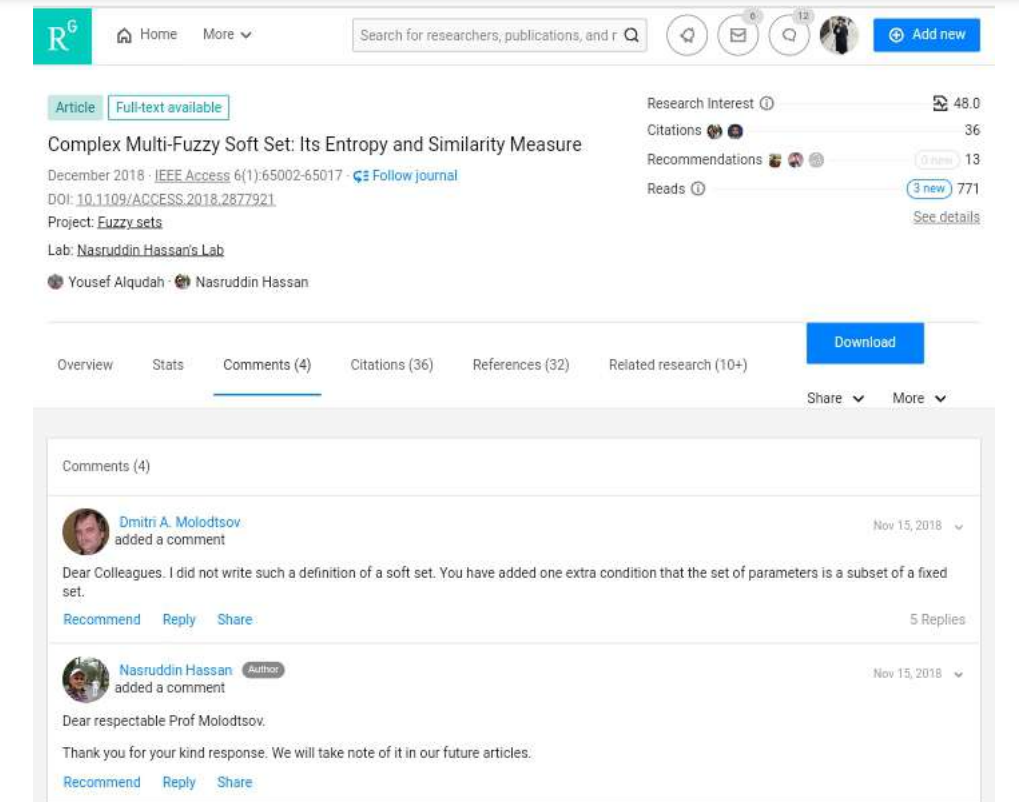}
    \caption{Comment of Prof. D. Molodtsov on ResearchGate}
    \label{fig:my_label}
\end{figure}\\ 
\noindent
In \cite{8}, the authors considered the sets of
attributes of the two soft sets for measuring similarity between them. But, it is not necessary that the sets of attributes should always contain some common
elements. In \cite{12}, Kamacı defined a formula for measuring similarity, which results, if there is no any common attribute between the two soft sets then the similarity between them is always 0. During  early research career in soft set theory and hybrid structures \cite{14, 15},  the first author of this paper also followed wrong notions of soft set and associated structures, which were developed by others other than Molodtsov. But, it is not a matter of shame for the first author as he learnt correct concepts of soft set theory and related structures while working with Molodtsov in \cite{16} and in other two papers. Since then, he has been trying to follow the correct paths in soft set theory as set by Molodtsov. \\ 

\noindent
In \cite{3}, Molodtsov mentioned, ``\textit{The principle of constructing correct operations for soft sets is very simple. Operations should be defined through families of sets $\tau(S,A)$}". Hence, the results of soft sets must be  dependent on the sets $\tau(F,A)$ and $\tau(G,B)$ of two soft sets $(F,A)$ and $(G,B)$ respectively, but not on the sets of attributes. In this paper, we try to propose various correct operations of soft sets accurately in matrix forms, as well as we find a similarity measure between two soft sets that is independent of the set of parameters. We also consider the condition where the set of attributes can have varying cardinality and no shared parameters, which follows philosophy of Molodtsov behind introducing soft set theory \cite{1,2,3,4}.
\section{Preliminaries} 
In this section, we discuss some preliminary notions of soft set theory and related results. Let $A$ be a set of parameters and $X$ be a universal set, over which a soft set is defined. The formal definition of soft set is as follows: \\
\begin{definition} \cite{1}
A pair $(S,A)$ will be called a soft set over $X$, if $S$ is a mapping from the set $A$ to the set of subsets of the set $X$, i.e., $S:A\rightarrow 2^X$. In fact, a soft set is a parametrized family of subsets.\\ \\
\noindent
If the soft set $(S,A)$ is given, then the family $\tau(S,A)$ can be defined as  $\tau(S,A)=\{S(a),a\in A\}$ \cite{3}.\\
\begin{definition}
\cite{3} Two soft sets $(S,A)$ and $(S',A')$ given over the universal set $X$ will be called equal and we write $(S,A)=(S',A')$ if and only if $S=S'$ and $A=A'$.
\end{definition}

\begin{definition}\cite{3}
Two soft set $(S,A)$ and $(S',A')$ given over the universal set $X$ will be called equivalent and written as $(S,A)\cong (S',A')$ if and only if $\tau(S,A)=\tau(S',A')$.
\end{definition}
\end{definition}
\noindent It is easy to note that equivalence of soft set is an equivalence relation.
\begin{definition}\cite{3}
A soft set $(S,A)$ internally approximates a soft set $(F,D)$ denoted by $(S,A)\subseteq (F,D)$ if for any $d\in D$ such that $F(d)\neq \phi$, there exists $a\in A$ for which $\phi \neq S(a)\subseteq F(D)$.
\end{definition}
\begin{definition}\cite{3}
A soft set $(S,A)$ externally approximates a soft set $(F,D)$ denoted by $(S,A)\supseteq (F,D)$ if for any $d\in D$ such that $F(d)\neq X$, there exists $a\in A$ for which $X \neq S(a)\supseteq F(D)$.
\end{definition}
\begin{definition}\cite{3}
If $(S,A)\subseteq (F,D)$ but the relation $(F,D)\subseteq (S,A)$ has no place, then we say that the soft set $(S,A)$ internally strictly approximate the soft set $(F,D)$, denoted by $(S,A)\subset (F,D)$.
\end{definition}
\begin{definition}\cite{3}
If $(S,A)\supseteq (F,D)$ but the relation $(F,D)\supseteq (S,A)$ has no place, then we say that the soft set $(S,A)$ outwardly strictly approximate the soft set $(F,D)$, denoted by $(S,A)\supset (F,D)$.
\end{definition}
\begin{definition}\cite{3}
The soft set $(S,A)$ is internally equivalent to the soft set $(F,D)$, denoted by $(S,A) \begin{matrix}
\subset \\
\approx
\end{matrix} (F,D)$ if $(S,A)\subseteq (F,D)$ and $(F,D)\subseteq (S,A)$.
\end{definition}

\begin{definition}\cite{3}
The soft set $(S,A)$ is externally equivalent to the soft set $(F,D)$, denoted by $(S,A) \begin{matrix}
\supset \\
\approx
\end{matrix} (F,D)$ if $(S,A)\supseteq (F,D)$ and $(F,D)\supseteq (S,A)$.
\end{definition}

\begin{definition}\cite{3}
The soft set $(S,A)$ is weakly equivalent to the soft set $(F,D)$, denoted by $(S,A)\approx (F,D)$ if $(S,A) \begin{matrix}
\subset \\
\approx
\end{matrix} (F,D)$ and $(S,A) \begin{matrix}
\supset \\
\approx
\end{matrix} (F,D)$.
\end{definition}

\begin{definition}\cite{3}
The minimal and maximal on inclusion of sets of the family $\tau(S,A)$ is defined as follows.
\begin{center}

$MIN(\tau(S,A))=\{B\in \tau(S,A)| B\neq \phi, \nexists B'\in \tau(S,A):B'\subset B\neq B'\neq \phi \}$\\
$MAX(\tau(S,A))=\{B\in \tau(S,A)| B\neq X, \nexists B'\in \tau(S,A):B'\supset B\neq B'\neq X \}$

\end{center}

\end{definition}

\begin{definition}\cite{3}
A relation $\Omega$ is called correct if for any quadruple of pairwise equivalent soft sets $(S,A)\cong (S',A')$, $(F,D)\cong (F',D')$, given over the universal set $X$, the following equality is satisfied.
\begin{center}
    $\Omega ((S,A),(F,D))=\Omega((S',A'),(F',D'))$
\end{center}
\end{definition}

\begin{definition}\cite{3}
The unary operation complement of $(S,A)$,  $C(S,A)=(W,A)$ is defined as follows. The set of parameters remains the same and the mapping is given by $W(a)=X\setminus S(a)$ for any $a\in A$. 
\end{definition}
\begin{definition}\cite{3}
The binary operation union $(S,A)\cup (F,D)=(H,A\times D)$ for a pair of soft sets $(S,A)$ and $(F,D)$ given over a universal set $X$ is defined as follows: The parameter set is chosen equal to the direct product of the  parameter sets, i.e., equal to $A\times D$, and the corresponding mappings are given by the formula $H(a,d)=S(a)\cup F(d)$, $(a,d)\in A\times D$.
\end{definition}
\begin{definition}\cite{3} 
The binary operation intersection $(S,A)\cap (F,D)= (W,A\times D)$  for a pair of soft sets $(S,A)$ and $(F,D)$ given over a universal set $X$ is defined as follows: The parameter set is chosen equal to the direct product of the parameter sets, i.e., equal to $A\times D$, and the corresponding mappings are given by the formula $W(a,d)=S(a)\cap F(d)$, $(a,d)\in A\times D$.
\end{definition}
\begin{definition}
The binary operation product $(S,A)\times (F,D)=(X, A\times D)$ for a pair of soft sets $(S,A)$ and $(F,D)$ given over a universal set $X$ is defined as follows. The parameter set is chosen equal to the direct product of the parameter sets, i.e., equal to $A\times D$ and the corresponding mappings are given by the formula $X(a,d)=S(a) \times F(d), (a,d)\in A\times D$.
\end{definition}
\noindent
Now, we  consider an illustrative example to discuss the above operations.
\begin{example}
Let us consider two soft sets $(F, A)$ and $(G, B)$ over a universal set $X$, where $A$ and $B$ be two sets of attributes. We consider $X=\{a, b, c\}$,
$A$=$\{x, y, z\}$, and
$B$=$\{m,n,o\}$. We  define the soft set $(F,A)$ as follows:\\ \begin{center}
    $F(x)=\{b,c\}$, $F(y)=\{c\}$, $F(z)=\{a\}$.
\end{center}
Hence, $\tau(F,A)=\{\{b,c\},\{c\},\{a\}\}$.
Similarly, we define the soft set $(G,B)$ as follows:\\
\begin{center}
    $G(m)=\{a\}$, $G(n)=\{c\}$, $G(o)=\{c\}$.
\end{center}
\noindent
Hence, $\tau(G,B)=\{\{a\},\{c\}\}$.\\
\noindent
Now. we consider the unary operation complement `$C$' of the soft set $(F,A)$ which is given by
$C(F,A)=(CF,A)$, where $CF(a)=X\setminus S(a),\forall a \in A$. So, $CF(x)=\{a\}$, $CF(y)=\{a,b\}$, $CF(z)=\{b,c\}$. Therefore, $\tau(CF,A)=\{\{a\},\{a,b\},\{b,c\}\}$.\\
\noindent
Again, we consider the binary operation union `$\bigcup$' between the soft sets $(F,A)$ and $(G,B)$ denoted by $(F,A)\bigcup(G,B)=(H,A\times B)$, where $A\times B$ is the Cartesian product of $A$ and $B$, and the mapping is given by $H(a,b)=F(a)\cup G(a)$, where $(a,b)\in A\times B$. Thus, 
$A\times B=\{\{x,m\},\{x,n\},\{x,o\},\{y,m\},\{y,n\},\{y,o\},\{z,m\},\{z,n\},\{z,o\}\}$. So, we obtain
$H(x,m)=\{a,b,c\},\\ H(x,n)=\{b,c\},H(x,o)=\{b,c\}, H(y,m)=\{a,c\}, H(y,n)=\{c\}, H(y,o)=\{c\}, H(z,m)=\{a\}, H(z,n)=\{a,c\}$, and $H(z,o)=\{a,c\}$.
Hence, $\tau(H,A\times B)=\{\{a,b,c\},\{b,c\},\{a,c\},\{c\},\{a\}\}$.\\
\noindent
Now, we also consider the binary operation intersection `$\bigcap$' between $(F,A)$ and $(G,B)$ denoted by $ (F,A)\bigcap (G,B)= (W,A\times B)$,  and the mapping is given by
$W(a,b)=F(a)\cap G(b)$, where $(a,b)\in A\times B$. Thus, we get $W(x,m)=\phi,W(x,n)=\{c\}, W(x,o)=\{c\}, W(y,m)=\phi, W(y,n)=\{c\}, W(y,o)=\{c\}, W(z,m)=\{a\},$ $W(z,n)$ =$\phi$, and  $W(z,o)=\phi$.
Therefore, $\tau(W,A\times B)=\{\{a\},\{c\}\}$.\\
\noindent
Also, we consider the binary operation product `$\times$' between $(F,A)$ and $(G,B)$, denoted by $(F,A)\times (G,B)=(X,A\times B)$,  and the mapping is given by
$X(a,b)=F(a) \times G(b),$ where $(a,b)\in A\times B$.
$W(x,m)=\{b,c\}\times \{a\}=\{(b,a),(c,a)\}$, $W(x,n)=\{b,c\}\times \{c\}=\{(b,c),(c,c)\}$,
$W(x,o)=\{b,c\}\times \{c\}=\{(b,c),(c,c)\}$, $W(y,m)=\{c\}\times \{a\}=\{(c,a)\}$, $W(y,n)=\{c\}\times \{c\}=\{(c,c)\}$, $W(y,o)=\{c\}\times \{c\}=\{(c,c)\}$, $W(z,m)=\{a\}\times \{a\}=\{(a,a)\}$, $W(z,n)=\{a\}\times \{c\}=\{(a,c)\}$, and $W(z,o)=\{a\}\times \{c\}=\{(a,c)\}$.
Hence, $\tau(X,A\times B)=\{\{(b,a),(c,a)\},\{(b,c),(c,c)\},\{(c,a)\},\{(c,c)\},\{(a,a)\},\{(a,c)\}\}$.

\end{example}

\section{Matrix representation of soft sets}
In this section, we shall discuss the matrix representation of a soft set. A soft set $(F,A)$ defined over a universal set $X$ can be represented by a matrix $M$ such that the number of rows of $M$ is equal to the cardinality of the universal set $X$, and the number of columns of $M$ is equal to the cardinality of the set of attribute $A$. \\ \\
\noindent
We consider the cardinalities of $X$ and  $A$  to be finite for the practical feasibility of computation and other real life purposes. Let us consider $|X|= n$ and $|A|=m$, where $X=\{x_1,x_2,x_3,\dots,x_n\}$ and $A=\{a_1,a_2,a_3,\dots ,a_m\}$. Then, the binary representation table of $(F,A)$ is given below.\\

\begin{table}[h]
\begin{center}
\begin{tabular}{c|c c c c c}
& $F(a_1)$ & $F(a_2)$ & $F(a_3)$ & $\dots$ & $F(a_m)$\\
\hline
$x_1$ & $m_{11}$ & $m_{12}$ & $m_{13}$ & $\dots$ & $m_{1m}$\\
$x_2$ & $m_{21}$ & $m_{22}$ & $m_{23}$ & $\dots$ & $m_{2m}$\\
$x_3$ & $m_{31}$ & $m_{32}$ & $m_{33}$ & $\dots$ & $m_{3m}$\\
\vdots & \vdots & \vdots & \vdots & $\ddots$ & \vdots\\
$x_n$  & $m_{n1}$ & $m_{n2}$ &  $m_{n3}$ & \dots & $m_{nm}$

\end{tabular}
\end{center} 
\caption{\label{tab1}Binary representation table of soft set $(F,A)$}
\end{table}
\noindent
Here,  
\begin{center}
$m_{ij}=\begin{cases} 
1;\: if \: x_i \in F(a_j),\\
0;\: if \: x_i \notin F(a_j).
\end{cases}$
\end{center}
Moreover, $i=1,2,3, \dots , n$ and $j=1,2,3, \dots, m$. 
\noindent
Thus, we get the following matrix $M$  as a matrix representation of $(F,A)$ from table \ref{tab1}. 
\begin{center}
$M=
\begin{pmatrix}
m_{11} & m_{12} & m_{13} & \dots & m_{1m}\\
m_{21} & m_{22} & m_{23} & \dots& m_{2m}\\
m_{31} & m_{32} & m_{33} & \dots& m_{3m}\\
\vdots & \vdots & \vdots & \ddots & \vdots\\
m_{n1} & m_{n2} &  m_{n3} & ...& m_{nm}
\end{pmatrix}$

\end{center}
So, every soft set can be transformed into it's matrix form and if the matrix of a soft set is given, then, we can uniquely determine the soft set. Here, we should concern about the order of the elements of $X$ in case of matrix representation, i.e., two different orderings of the elements of the set $X$ may led to two different representations of the same soft set, which may cause difficulties to define various operations related to the matrix representation. Hence, the order of $x_i$'s must be mentioned in case of matrix representation, where $x_i \in X$. Let us consider example 2.1 to illustrate the matrix representation of soft set $(F,A)$.
\begin{example}
Given, $X=\{a, b, c\}$, 
$A=\{x, y, z\}$,
$F(x)=\{b,c\}$, $F(y)=\{c\}$, and $F(z)=\{a\}$. Then, the binary representation of $(F,A)$ is given below.\\

\begin{center}
    
$M=$
\begin{tabular}{c|c c c }

& $F(x)$ & $F(y)$ & $F(z)$ \\
\hline
$a$ & $0$ & $0$ & $1$\\
$b$ & $1$ & $0$ & $0$ \\
$c$ & $1$ & $1$ & $0$\\

\end{tabular} \\
\end{center}
Hence, we obtain the following matrix $M$ of the soft set $(F,A)$. 

\begin{center}
$M=
\begin{pmatrix}
0 & 0 & 1\\
1 & 0 & 0\\
1 & 1 & 0
\end{pmatrix}$
\end{center}

\end{example}
\section{Operations on soft sets in matrix form}
In this section, we  formulate four  operations viz. complement, union, intersection and product of soft sets in matrix form.
\subsection{Soft complement}
Let $M=(x_{ij})_{m\times n}$ be the matrix representation of the soft set $(F,A)$  defined over a universal set $X$ and order of $M$ is $m\times n$. Then, $M'=(x'_{ij})_{m\times n}$ is the matrix representation of $C(F,A)$,  the complement of the soft set $(F,A)$ and it is defined as below.\\
\begin{center}
    
$M' =
\left\{
	\begin{array}{ll}
		x'_{ij}=1  &,\mbox{if } x_{ij}=0,\\
		x'_{ij}=0 &,\mbox{if } x_{ij}=1.
	\end{array}
\right.$
\end{center}
Here, $i=1,2,\dots ,m$, and $j=1,2,\dots ,n$. 
It's easy to see that order of the matrix $M'$ is identical to that of order of   matrix $M$. 

\subsection{Soft union}
Let, $K=(x_{ij})_{m\times n}$ and $L=(y_{ij})_{m\times p}$ be  matrix representations of  soft sets $(F,A)$ and $(G,B)$ respectively, defined over a universal set $X$. Since both the soft 
sets are defined over $X$, hence the matrix $K$ and $L$ may have  orders $m\times n$ and $m\times p$ respectively, where $m$, $n$ and $p$ are positive integers and $|n-p|\ge 0$. Now, we  represent  union of the two soft sets $(F,A)$ and $(G,B)$ by matrix representation and we consider $M$ to be the matrix of $(F,A)\cup (G,B)=(H, A\times B)$. Then, the  matrix $M$ is of  order $m\times (n.p)$. Let us define the matrix $M=(m_{ij})_{m\times (n.p)}$ as follows.\\ \\

$M=\begin{cases} 
m_{ij}=max\{x_{i1},y_{ij}\}, where \: i=1,2,\dots , m\: and \:j=1,2,\dots ,p\\
m_{ij}=max\{x_{i2},y_{i(j-p)}\}, where \: i=1,2,\dots , m\: and \:j=p+1,p+2,\dots ,2p\\
m_{ij}=max\{x_{i3},y_{i(j-2p)}\}, where \: i=1,2,\dots , m \: and \: j=2p+1,2p+2,\dots ,3p\\
...................................................................................................................\\
m_{ij}=max\{x_{in},y_{i(j-(n-1)p)}\}, where \: i=1,2,\dots , m \: and \:j=(n-1)p+1,(n-1)p+2,\dots ,n.p

\end{cases}$

The matrix $M$  follows the definition of union of two soft sets $(F,A)$ and $(G,B)$ given in \cite{3}.
\subsection{Soft intersection}
Let, $K=(x_{ij})_{m\times n}$ and $L=(y_{ij})_{m\times p}$ be matrix representations of two soft sets $(F,A)$ and $(G,B)$ respectively, defined over a universal set $X$. Since both the soft sets are defined over $X$, hence the matrix $K$ and $L$ may have the orders $m\times n$ and $m\times p$, where $m$, $n$ and $p$ are positive integers and $|n-p|\ge 0$.
Now, we represent  intersection of the two soft sets $(F,A)$ and $(G,B)$ by matrix representation and we consider $N$ to be the matrix of $(F,A)\cap (G,B)=(H, A\times B)$. Then, the matrix $N$ is of order $m\times (n.p)$. Let us define the matrix $N=(n_{ij})_{m\times (n.p)}$ as follows:\\ \\

$N=\begin{cases} 
n_{ij}=min\{x_{i1},y_{ij}\}, where \: i=1,2,\dots , m\: and \:j=1,2,\dots ,p\\
n_{ij}=min\{x_{i2},y_{i(j-p)}\}, where \: i=1,2,\dots , m\: and \:j=p+1,p+2,\dots ,2p\\
n_{ij}=min\{x_{i3},y_{i(j-2p)}\}, where \: i=1,2,\dots , m \: and \: j=2p+1,2p+2,\dots ,3p\\
...................................................................................................................\\
n_{ij}=min\{x_{in},y_{i(j-(n-1)p)}\}, where \: i=1,2,\dots , m \: and \: j=(n-1)p+1,(n-1)p+2,\dots ,n.p

\end{cases}$
The matrix $N$  also follows the definition of intersection of two soft sets $(F,A)$ and $(G,B)$ given in \cite{3}.
\subsection{Soft product}
Let $(F,A)$ and $(G,B)$ be two soft sets defined over a universal set $X$, where $X=\{x_1, x_2, \dots x_m\}$, $A=\{e_1,e_2,\dots e_n\}$, and $B=\{g_1,g_2,\dots g_p\}$. Then, we define the mapping $prod: (X\times X)\times (A\times B)\rightarrow\{0,1\}$ such that,
\begin{center}
$prod\{(x_k,x_j),(e_i,g_j)\}=\begin{cases} 
1, \: if \: x_k\in F(e_i)\: and \: x_j\in G(g_j)\\
0, \: otherwise\\

\end{cases}$\\
\end{center}
Here, $k=1,2,\dots m$; $i=1,2,\dots n$, and $j=1,2,\dots p$
Then, the matrix representation of $(F,A)\times (G,B)$ is given by $P=(p_{ij})_{m^2\times (n.p)}$ such that,
\begin{center}

$P=\begin{cases} 
p_{kj}=1, \: if \: prod\{(x_k,x_j),(e_i,g_j)\} =1\\
p_{kj}=0, \: if \: prod\{(x_k,x_j),(e_i,g_j)\} =0\\

\end{cases}$\\

\end{center}
\noindent
Now, we consider some examples to illustrate the matrix representations with the operators complement, union, intersection and product of soft sets. 
\begin{example}
From example 3.1, the matrix $M'$ of $C(F,A)$ can be found below.
\begin{center}
$M'=
\begin{pmatrix}
1 & 1 & 0\\
0 & 1 & 1\\
0 & 0 & 1
\end{pmatrix}$
\end{center}

\end{example}
\begin{example}
Let us consider union of two soft sets $(F,A)$ and $(G,B)$ defined over a universal set $X$. Let $X=\{a, b, c\}$, $A=\{x, y, z\}$, and $B=\{m,n,o,p\}$.
We define the soft set $(F, A)$ as follows:
$F(x)=\{b,c\}$, $F(y)=\{c\}$, $F(z)=\{a\}$.
Similarly, we define the soft set $(G,B)$ as follows:
$G(m)=\{a\}$, $G(n)=\{c\}$, $G(o)=\{c\}$, $G(p)=\{a,c\}$.
Then, the set of  parameters for $(F,A)\bigcup (G,B)$ is $A\times B=\{\{x,m\},\{x,n\},\{x,o\},\{x,p\},\{y,m\},\{y,n\},\{y,o\},\{y,p\},\{z,m\},\{z,n\},\{z,o\}, \{z,p\}\}$. Moreover, we have the  matrices $M=(x_{ij})_{3\times 3}$ and $N=(y_{ij})_{3\times 4}$ for $(F,A)$ and $(G,B)$ respectively as shown below.
\begin{center}
$M=
\begin{pmatrix}
0 & 0 & 1\\
1 & 0 & 0\\
1 & 1 & 0
\end{pmatrix}$,
$N=
\begin{pmatrix}
1 & 0 & 0 & 1\\
0 & 0 & 0 & 0\\
0 & 1 & 1 & 1
\end{pmatrix}$
\end{center}
Now, the matrix representation of $(F,A)\bigcup (G,B)$ is obtained as a matrix $U=(m_{ij})_{3 \times 12}$, which is shown below. We can easily calculate $m_{ij}$ for $i=1,2,3$ and $j=1,2,3,4$. Thus, we obtain the matrix $U$ as shown below. \\
\begin{center}

$U=
\begin{pmatrix}
1 & 0 & 0 & 1 & 1 & 0 & 0 & 1 & 1 & 1 & 1 & 1\\
1 & 1 & 1 & 1 & 0 & 0 & 0 & 0 & 0 & 0 & 0 & 0\\
1 & 1 & 1 & 1 & 1 & 1 & 1 & 1 & 0 & 1 & 1 & 1

\end{pmatrix}$

\end{center}
The matrix $U$ is identical to the matrix  obtained by calculating the union of the soft sets $(F,A)$ and $(G,B)$ as in definition 1.5 and then constructing the matrix from the obtained soft set. \\
Similarly, the intersection between $(F,A)$ and $(G,B)$ can also be calculated using 3.3.
\end{example}

\begin{example}
Let $(F,A)$ and $(G,B)$ be two soft sets defined over a universal set $X$. Also, we consider $X=\{a,b,c\}$, $A=\{m,n\}$, and $B=\{x,y\}$. Let the binary representations of $(F,A)$ and $(G,B)$ be $M$ and $N$ respectively.
\begin{center}
   
$M=$
\begin{tabular}{c|c c c c c}

& $F(m)$ & $F(n)$ \\
\hline
$a$ & 1 & 0 \\
$b$ & 1 & 1 \\
$c$ & 0 & 0 \\

\end{tabular},
$N=$
\begin{tabular}{c|c c c c c}

& $G(x)$ & $G(y)$ \\
\hline
$a$ & 0 & 0 \\
$b$ & 1 & 0 \\
$c$ & 1 & 1 \\

\end{tabular}
\end{center}
Then, the binary representation of $(F,A)\times (G,B)=(H,A\times B)$ is given below.
\begin{center}
   
\begin{tabular}{c|c c c c c}

& $H(m,x)$ & $H(m,y)$ & $H(n,x)$ & $H(n,y)$ \\
\hline
$(a,a)$ & 0 & 0 & 0 & 0 \\
$(a,b)$ & 1 & 0 & 0 & 0 \\
$(a,c)$ & 1 & 1 & 0 & 0 \\
$(b,a)$ & 0 & 0 & 0 & 0 \\
$(b,b)$ & 1 & 0 & 1 & 0 \\
$(b,c)$ & 1 & 1 & 1 & 1 \\
$(c,a)$ & 0 & 0 & 0 & 0 \\
$(c,b)$ & 0 & 0 & 0 & 0 \\
$(c,c)$ & 0 & 0 & 0 & 0 \\

\end{tabular}

\end{center}
Hence, the matrix $P$ of $(F,A)\times (G,B)$ is given by
$P=
\begin{pmatrix}
0 & 0 & 0 & 0 \\
1 & 0 & 0 & 0 \\
1 & 1 & 0 & 0 \\
0 & 0 & 0 & 0 \\
1 & 0 & 1 & 0 \\
1 & 1 & 1 & 1 \\
0 & 0 & 0 & 0 \\
0 & 0 & 0 & 0 \\
0 & 0 & 0 & 0 
\end{pmatrix}$ .\\
\noindent

\end{example}
\noindent
 Now, we want to mention an important idea related to soft set theory. Molodtsov in \cite{2} wrote, \textit{``Of course, when you specify a soft set, you have some semantic interpretation of this soft set. However, the mathematical formalism of soft sets does not imply any semantic sense on family of subsets or on the parameters. The parameters serve only the purpose to indicate a specific subset \dots \dots \dots To determine topology we have to define only the family of vicinities of a point. No comparison of vicinities and no other properties of these subsets are needed. The situation is quite similar for soft sets, as the soft set is a family of vicinities of a point except that the initial point (as in topology) may not exist. Thus, the role of parameters in definition of soft sets is only auxiliary. Parameters are used only as names of subsets. Therefore, the introduction of the notion of equivalence of soft sets $(S,A)$ and $(S',A')$ should be based on equality of families of sets $\tau(S,A)$ and $\tau(S',A')$, but not on equality of point-to-set mappings $S$ and $S'$."} Similar assumption can also be found in \cite{4}. Hence, it may be noted that in case of union, intersection and product of soft sets $(F,A)$ and $(G,B)$, the set of parameters $A\times B$ and $B\times A$ are indistinguishable. Hence, we propose the following theorem.
\begin{theorem}
Let $(F,A)$, $(G,B)$ and $(H,C)$ be three soft sets defined over a universal set $X$, then the following properties hold.\\
(i) $C(C(F,A))=(F,A)$,\\
(ii) $(F,A)\cup (G,B)= (G,B)\cup (F,A)$,\\
(iii) $\{(F,A)\cup (G,B)\}\cup (H,C)= (F,A)\cup \{(G,B)\cup (H,C)\}$,\\
(iv) $(F,A)\cap (G,B)= (G,B)\cap (F,A)$,\\
(v) $\{(F,A)\cap (G,B)\}\cap (H,C)= (F,A)\cap \{(G,B)\cap (H,C)\}$.

\end{theorem}

\begin{proof}
\noindent
We only prove $(i)$ and $(ii)$. \\
(i) For the  complement $C(F,A)=(W,A)$ of $(F,A)$, the set of parameters remains the same and the mapping is given by $W(a)=X\setminus F(a)$, for $a\in A$. Again, taking complement of the set $(W,A)$, we get, $C(W,A)=(U,A)$ and the mapping is given by $U(a)=X\setminus W(a)=X\setminus \{X \setminus F(a)\}= F(a), $ for $a\in A$. Thus, $C(C(F,A))=(F,A)$.\\ \\
(ii) Let $(F,A)\cup (G,B)=(H,A\times B)$. Then, the set of parameters for $(F,A)\cup (G,B)$ is $A\times B$, and the corresponding mapping is given by $H(a,b)=F(a)\cup G(b)$, where $(a,b)\in A\times B$. Again, we consider $(G,B)\cup (F,A)=(I,B\times A)$. In this case, the set of parameter is $B\times A$, and the corresponding mapping is given by $I(b,a)=G(b)\cup F(a)$. Since, according to Molodtsov \cite{2}, the sets of parameters are auxiliaries in case of soft sets, thus $A\times B$ and $B\times A$ are indistinguishable in the sense that $(a,b)\in A\times B$ and $(b,a)\in B\times A$ are indistinguishable. So, $H(a,b)=I(b,a)$ because $F(a)\bigcap G(b)=G(b)\bigcap F(a).$ Hence, $(F,A)\cup (G,B)= (G,B)\cup (F,A)$.\\ 
   
\end{proof}

\section{Similarity between two soft sets}
In this section, we define similarity measure between two soft sets $(F,A)$ and $(G,B)$ defined over a universal set $X$. We consider the concept of matrix representation of a soft set while measuring the similarity between two soft sets. For real world phenomena, we consider both the universal set and the set of parameters for a soft set to be finite. 

\begin{definition}

Let $X$ be a universal set, $A$ and $B$ are two sets of attributes such that $|X|=m$, $|A|=n$ and $|B|=p$. Here, $|K|$ denotes the cardinality of a set  $K$. Without loss of generality, we consider $n\ge p$. Also, we assume $Y=(y_{ij})_{m\times n}$ and $Z=(z_{ij})_{m\times p}$ to be the matrices of $(F,A)$ and $(G,B)$ respectively. Now, we consider following two cases.
\subsection*{case 1 : when $n=p$}
If $n=p$, then we construct a new matrix $D=(d_{ij})$ which depends on matrices $Y$ and $Z$ in the following way:
\begin{center}
    
$D=\begin{cases}
d_{ij}=1, \ if \ y_{ij}=z_{ij}\\
d_{ij}=0, \ if \ y_{ij}\neq z_{ij}
\end{cases}$

\end{center}
Then, similarity measure between the two soft sets $(F,A)$ and $(G,B)$  is denoted by $Sim\{(F,A),(G,B)\} $ and defined as:
\begin{center}
    $Sim\{(F,A),(G,B)\}=\frac{{\sum_{i,j}}\{ d_{ij} \ :\ d_{ij}=1\}}{|X|\times|A|}$.
\end{center} 

\subsection*{case 1 : when $n>p$}
Suppose $n-p=t$, i.e., the matrix $Z=(z_{ij})_{m\times p}$ has exactly $t$ columns more than the matrix $Y=(y_{ij})_{m\times n}$. Now, we first consider the submatrix $W=(w_{ij})_{m\times n}$ of $Z$ with the first  $n$ columns of $Z$ and with all the rows of $Z$ i.e., with $m$ rows. Now, the matrices $Y$ and $W$ are of the same order. Again, let us define a new matrix $D=(d_{ij})_{m\times n}$ as follows.
\begin{center}
    
$D=\begin{cases}
d_{ij}=1, \ if \ y_{ij}=w_{ij}\\
d_{ij}=0, \ if \ y_{ij}\neq w_{ij}
\end{cases}$
\end{center}
\noindent
Next, we consider the remaining submatrix of $Z$ with $m$ rows and $(p-t)$ columns beginning from the $(t+1)^{th}$ column up to the $p^{th}$ column. We  denote it by $M=(m_{ij})_{m\times (p-t)}$, which is of order $m\times (p-t)$. Now, we construct a zero matrix $N=(n_{ij})_{m\times (p-t)}$, in which all the entries are 0 and is of order $m\times (p-t)$. Again, we define the matrix $C=(c_{ij})_{m\times (p-t)}$ as follows. 
\begin{center}
    
$C=\begin{cases}
c_{ij}=1, \ if \ m_{ij}=n_{ij}=0,\\
c_{ij}=0, \ otherwise.
\end{cases}$
\end{center}
Then, similarity measure between the two soft sets $(F,A)$ and $(G,B)$ is defined as below.
\begin{center}
    $Sim\{(F,A),(G,B)\}=\frac{\sum_{i,j}\{ d_{ij} \ :\ d_{ij}=1\}+ \sum_{i,j}\ \{ c_{ij}\ : \ c_{ij}=1\}}{|X|\times max\{|A|,|B|\}}$
\end{center}

\end{definition}
\noindent
These two similarity measures are independent of the set of attributes chosen and they only depend on the families $\tau(F,A)$ and $\tau(G,B)$. 
Two soft sets is said to be completely similar if their similarity measure is 1 or completely dissimilar if their similarity measure is 0. Thus, we obtain similarity measures using Molodtsov's ideas \cite{2} as stated above.
Let us consider an illustrative example.
\begin{example}
Let $M$ and $N$ be two matrices of the soft sets $(F,A)$ and $(G,B)$.
\begin{center}
    
$M=
\begin{pmatrix}
1 & 0 & 1\\
1 & 0 & 0\\
1 & 0 & 1
\end{pmatrix}$
$N=
\begin{pmatrix}
0 & 1 & 1 & 0\\
1 & 0 & 0 & 1\\
1 & 1 & 1 & 0
\end{pmatrix}$

\end{center}
\noindent 
Thus,
$d_{11}=0$,
$d_{12}=0$, 
$d_{13}=1$, 
$d_{21}=1$, 
$d_{22}=1$, 
$d_{23}=1$
$d_{31}=1$, 
$d_{32}=0$, and 
$d_{33}=1$. 
Now, we construct the matrix $D=(d_{ij})_{3 \times 3}$, $(i=1,2,3)$, $(j=1,2,3)$ of order $3\times 3$ as follows.
\begin{center}
    
$D=
\begin{pmatrix}
0 & 0 & 1\\
1 & 1 & 1\\
1 & 0 & 1
\end{pmatrix}$
\end{center}
Next, we construct a zero matrix $R=(r_{ij})_{3 \times 1}$, and a column matrix $S=(s_{ij})_{3 \times 1}$, each of order $3\times 1$, where $i=1$, and $j=1,2,3$. It is easy to observe that $S$ is the submatrix of $N$  consisting of $4^{th}$ column of $N$. 
\begin{center}
    
$R=
\begin{pmatrix}
0 \\
0\\
0 
\end{pmatrix}$,
$S=
\begin{pmatrix}
0 \\
1 \\
0 
\end{pmatrix}$
\end{center}
Thus,
$c_{11}=1$,
$c_{21}=0$, and
$c_{31}=1$.
Now, construct the matrix $C$ as follows.\\

\begin{center}
    
$C=
\begin{pmatrix}
1 \\
0 \\
1 
\end{pmatrix}$

\end{center}
Hence,
\begin{center}
    $Sim\{(F,A),(G,B)\}=\frac{\sum_{i,j}\{ d_{ij} \ :\ d_{ij}=1\}+ \sum_{i,j}\ \{ c_{ij}\ : \ c_{ij}=1\}}{|X| \times max\{|A|,|B|\}}= \frac{2}{3}$.
    
\end{center}

\end{example}

\begin{theorem}
Let $(F,A)$ and $(G,B)$ be two soft sets over a universal set $X$. Then, the following results hold:\\
(i) $0\le Sim\{(F,A),(G,B)\}\le 1$\\
    (ii) $Sim\{(F,A),(G,B)\}=Sim\{(G,B),(F,A)\}$\\
   (iii) $Sim\{(F,A),(F,A)\}=1$

\end{theorem}
\begin{proof} (i) Let, $|X|=m$, $|A|=n$  and $|B|=p$. Without loss of generality, we consider $n\ge p$. Let, $Y=(y_{ij})_{m \times n}$ and $Z=(z_{ij})_{m \times p}$ be the matrices of the soft sets $(F,A)$ and $(G,B)$ respectively. We consider following two cases.  \\ \\
\textbf{case 1($n=p$)}\\ \\

\begin{center}

$\sum_{i,j}\{ d_{ij} \ :\ d_{ij}=1\}\le |X|\times |A|$\\

$\implies \frac{\sum_{i,j}\{ d_{ij} \ :\ d_{ij}=1\}}{|X|\times |A|}\le 1$\\ 

$\implies Sim\{(F,A),(G,B)\}\le 1$

\end{center}

Also,

\begin{center}
$\sum_{i,j}\{ d_{ij} \ :\ d_{ij}=1\}\ge 0$\\

$\implies \frac{\sum_{i,j}\{ d_{ij} \ :\ d_{ij}=1\}}{|X|\times |A|}\ge 0$\\

$\implies Sim\{(F,A),(G,B)\}\ge 0$\\ 
\end{center}

\noindent     
\textbf{case 2($n>p$)}

\begin{center}

$\sum_{i,j}\{ d_{ij} \ :\ d_{ij}=1\}+ \sum_{i,j}\ \{ c_{ij}\ : \ c_{ij}=1\}\le |X|\times max\{|A|,|B|\}$\\

$\implies \frac{\sum_{i,j}\{ d_{ij} \ :\ d_{ij}=1\}+ \sum_{i,j}\ \{ c_{ij}\ : \ c_{ij}=1\}}{|X|\times max\{|A|,|B|\}}\le 1$\\

$\implies Sim\{(F,A),(G,B)\}\le 1$\\

\end{center}
Also, \\ 
\begin{center}

$\sum_{i,j}\{ d_{ij} \ :\ d_{ij}=1\}+ \sum_{i,j}\ \{ c_{ij}\ : \ c_{ij}=1\}\ge 0$\\

$\implies \frac{\sum_{i,j}\{ d_{ij} \ :\ d_{ij}=1\}+ \sum_{i,j}\ \{ c_{ij}\ : \ c_{ij}=1\}}{|X|\times max\{|A|,|B|\}}\ge 0$\\ 

$\implies Sim\{(F,A),(G,B)\}\ge 0$\\

\end{center}
Thus, combining the above results we get  $0\le Sim\{(F,A),(G,B)\}\le 1$.\\ \\
(ii) Let $M$ and $N$ be the matrices of order $m\times n$ and $m\times p$  of the soft sets $(F,A)$ and $(G,B)$  respectively. If $n=p$, then the result is obvious.\\
If $n>p$ and $n-p=t$, then we construct a zero matrix $N$ of order $m\times (n-t)$ to obtain the matrix $C$ as defined in definition 5.1. If $n<p$ and $p-n=t$, then similarly, we  construct the zero matrix $N'$ of order $m\times (p-t)$ to obtain the matrix $C$. In both the cases, we obtain the same matrix $C$. Hence, it  does not affect if we interchange the order of the soft sets for which we have to calculate the similarity measure.\\ \\
(iii) Directly follows from definition 5.1.
\end{proof}

\begin{theorem}

If $(F,A)$ be a soft set defined over a universal set $X$ and  $(CF,A)$ denotes the complement of the soft set $(F,A)$, then $Sim\{(F,A),(CF,A)\}=0$.
\end{theorem}
\begin{proof}
The proof is obvious.
\end{proof}

\begin{theorem}
Let  $(F,A)$ and $(G,B)$ be two soft sets defined over a universal set $X$. If the mappings $F$ and $G$ are one-one and $(F,A)\cong (G,B)$, then, $Sim\{(F,A),(G,B)\}=1$, provided the ordering of the attributes is independent of choice.
\end{theorem}
\begin{proof}
For two soft sets $(F,A)$ and $(G,B)$ defined over a universal set $X$, $(F,A)\cong (G,B)$ implies $\tau(F,A)$ = $\tau(G,B)$. Since $\tau(F,A)=\tau(G,B)$, then the mappings $F$ and $G$ are one-one imply that the $|A|=|B|$. It implies that the matrices of both the soft sets will be of same order. Now, an element of $\tau(F,A)$  represents one column in the respective matrix of $(F,A)$ and we obtain a matrix similarly for $(G,B)$. It is also given that the ordering of the attributes is independent of choice. Hence, $\tau(F,A)$ = $\tau(G,B)$ implies that all the columns in the matrices of the two soft sets are identical. Thus, $Sim\{(F,A),(G,B)\}=1$.
\end{proof}
\noindent
Now, we consider a case where $\tau(F,A)\cap \tau (G,B)= \phi$, for two soft sets $(F,A)$ and $(G,B)$. For the practical feasibility, let us discuss a situation by a real life example. Granular computing is an emerging computing paradigm of information processing that concerns the processing of complex information entities called ``information granules", which arise in the process of data abstraction and derivation of knowledge from information or data \cite{18}.

\begin{figure}
    \centering
    \includegraphics[width=14cm, height=8cm]{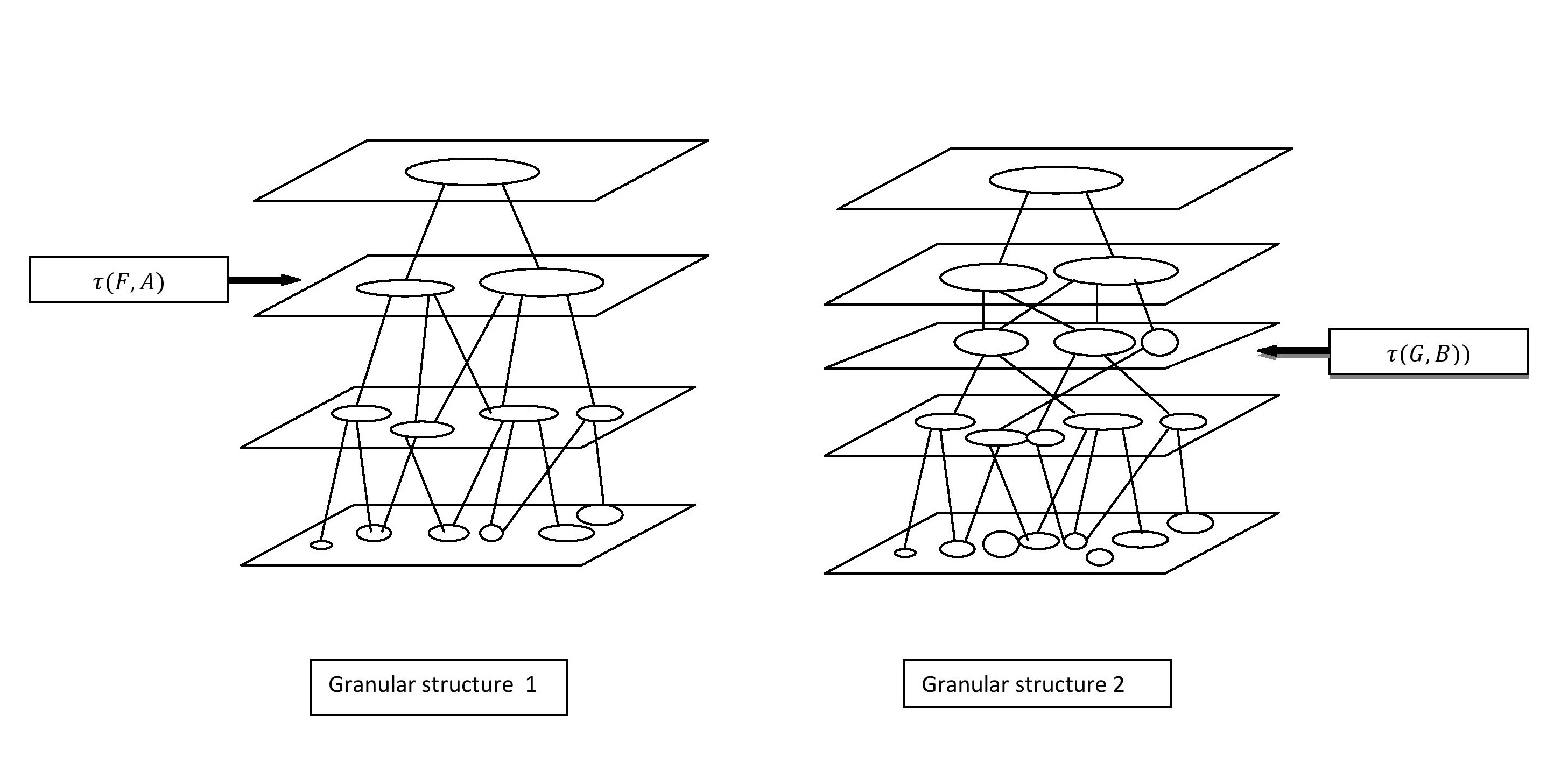}
    \caption{Representations of two granular structures}
    \label{1}
\end{figure}

\noindent
In this process, there are granules and several granular layers in a granular structure. A granule may be a subset, class, object or cluster of a universe \cite{17}. In granular computing \cite{18, 19}, we may construct a granular layer by using  a soft set say $(F,A)$ and the granules of the layer may be represented by the elements of  $\tau(F,A)$ of the soft set $(F,A)$. Then, we can have similarity measures   between the different layers formed by different soft sets. It is often feasible to find that the information in a specific layer of a granular structure does not match with an another layer of a different granular structure. In spite of it, we cannot conclude that there is no similarity between the two layers of different granules. For example, we consider two granular layers as shown in figure \ref{1} represented by   two soft sets $(F,A)$ and $(G,B)$ defined over the same universe  $X=\{a,b,c\}$. The elements in the families $\tau(F,A)$ and $\tau(G,B)$ will represent the granules in two granular layers. Suppose $\tau(F,A)=\{\{a,b\},\{c\}\}$ and $\tau(G,B)=\{\{a,c\},\{b\},
\{d,e\}\}$ are representing  the sets of two and three granules in granular structure 1 and granular structure 2 respectively as shown in the figure \ref{1}. Here, $\tau(F,A)\cap \tau (G,B)= \phi$, but we can not say that there is no similarity between the granular layers because the elements $a$, $b$ and $c$ are in some granules in both the layers. From this discussions, we conclude the following theorem which may be useful for granular computing.\\

\begin{theorem}

If $(F,A)$ and $(G,B)$ be two soft sets defined over a universal set $X$, then $\tau(F,A)\cap \tau (G,B)= \phi$ does not imply that $Sim\{(F,A), (G,B)\}=0$ in general.
\end{theorem}
\noindent
Let us consider an example to illustrate the above theorem.
\begin{example}
Let $(F,A)$ and $(G,B)$ be two soft sets representing two granular layers in two different granular structures defined over a universal set $X$.
Let, $X=\{a,b,c,d,e\}$, $A=\{m,n\}$, and $B=\{x,y\}$.
The mappings $F$ and $G$ are given by
$F(m)=\{a,b\}, F(n)=\{e\},$
$G(x)=\{b,c\}$, and $ G(y)=\{c,d,e\}$.
Hence, $\tau(F,A)=\{\{a,b\},\{e\}\}$ and
$\tau(G,B)=\{\{b,c\},\{c,d,e\}\}$.
In these two families, each element will represent a granule in the granular layer of the respective granular system.
We find that $\tau(F,A)\cap \tau(G,B)=\phi$. It is now easy to find that
$Sim\{(F,A),(G,B)\}=\frac{3}{5}$.

\end{example}

\noindent
Now, let us now consider some properties of soft sets with respect to the internal and external approximations  of Molodtsov \cite{3}.
\begin{definition}
Let $M$ be  a matrix of a soft set $(F,A)$ defined over a universal set $X$. We consider $A=\{e_{1},e_{2}, \dots, e_{n}\}$, i.e., $|A|=n$. Then, gravity of an attribute $e_i\in A$ is  denoted by $\chi (e_1)$ and it is defined  to be the total number of 1's in the ordered $i^{th}$ column of the  matrix $M$ of $(F,A)$, where, $i=1,2, \dots, n$.
\end{definition}

\begin{theorem}
Consider two soft sets $(F,A)$ and $(G,B)$ defined over a universal set $X$ such that $(F,A)\subseteq (G,B)$. Let $|A|=m$ and $|B|=n$. Then, $\chi(e_i) \leq \chi (g_j)$  in some order, where $i=1,2,\dots m$ and $j=1,2,\dots n$.
\end{theorem}
\begin{corollary}
If $(F,A) \subseteq (G,B)$ and $|A|=|B|=n$, then $\chi(e_i) \leq \chi(g_i)$, where $i=1,2,\dots n$, if the ordering of attributes is independent of choice.
\end{corollary} 
\noindent
Let us consider an illustrative example.
\begin{example}
Let $(F,A)$ and $(G,B)$ be two soft set defined over a universal set $X$. Let $A=\{e_1,e_2,e_3\}$ and $B=\{g_1,g_2,g_3,g_4\}$. If $M$ and $N$ are the matrices of $(F,A)$ and $(G,B)$ respectively, where,

\begin{center}
    
$M=
\begin{pmatrix}
1 & 1 & 1\\
1 & 1 & 0\\
0 & 1 & 0
\end{pmatrix}$,
$N=
\begin{pmatrix}
1 & 1 & 1 & 1 \\
1 & 0 & 0 & 1\\
0 & 0 & 1 & 1
\end{pmatrix}$

\end{center}
then obviously, $(F,A)\subseteq (G,B)$.
Hence, $\chi(e_1)= \chi(g_1)$, $\chi(e_2)= \chi(g_4)$, $\chi(e_3)\leq \chi(g_3)$.

\end{example}
\begin{corollary}
Let $(F,A)$ and $(G,B)$ be two soft set defined over a universal set $X$ and $(F,A)\subseteq (G,B)$. Let $A=\{e_1,e_2,\dots, e_m\}$ and $B=\{g_1,g_2,\dots, g_n\}$. Then, it is not always true that $\sum_{i=1}^m \chi(e_i) \leq \sum_{j=1}^n \chi (g_j)$.
\end{corollary}
\noindent
The following example is used to support the preceding corollary.
\begin{example}
Let $(F,A)$ and $(G,B)$ be two soft set defined over a universal set $X$. Let $A=\{e_1,e_2,e_3\}$ and $B=\{g_1,g_2,g_3,g_4,g_5\}$. If $M$ and $N$ are the matrices of $(F,A)$ and $(G,B)$ respectively, where,

\begin{center}
$M=
\begin{pmatrix}
1 & 1 & 1 & 1 & 1 \\
0 & 0 & 0 & 0 & 0 \\
0 & 0 & 0 & 0 & 0
\end{pmatrix}$,
$N=
\begin{pmatrix}
1 & 1 & 1\\
0 & 0 & 0\\
0 & 1 & 0
\end{pmatrix}$

\end{center}
then obviously, $(F,A)\subseteq (G,B)$.
But,  $\sum_{i=1}^3 \chi(e_i)=5 \nleq \sum_{j=1}^5 \chi (g_j)=4$.
\end{example}

\begin{theorem}
Let  $(F,A)$ and $(G,B)$ be two soft sets defined over a universal set $X$ such that $(F,A)\supseteq (G,B)$. Let $|A|=m$ and $|B|=n$. Then, $\chi(e_i)  \geq \chi (g_j) $ in some order, where $i=1,2,\dots m$ and $j=1,2,\dots n $.
\end{theorem}

\begin{corollary}
If $(F,A) \supseteq (G,B)$ and $|A|=|B|=n$, then $\chi(e_i) \geq \chi(g_i)$, where $i=1,2,\dots n$, if the ordering of attributes is independent of choice.
\end{corollary}
\begin{corollary}
Let $(F,A)$ and $(G,B)$ be two soft sets defined over a universal set $X$ and $(F,A)\supseteq (G,B)$. Let $A=\{e_1,e_2,\dots e_m\}$ and $B=\{g_1,g_2,\dots g_n\}$. Then, it is not always true that $\sum_{i=1}^m \chi(e_i) \geq \sum_{j=1}^n \chi (g_j)$.
\end{corollary}
\begin{theorem}
Let $(F,A)$ and $(G,A)$ be two soft sets defined over a universal set $X$. If $(F,A) \begin{matrix}
\subset \\
\approx
\end{matrix} (G,A)$, then $Sim\{(F,A),(G,A)\}=1$, for some specific ordering of the attributes in the matrix representation of the soft sets.
\end{theorem}
\begin{proof}
Since the sets of attributes are same for both of the soft sets $(F,A)$ and $(G,A)$, hence the matrices are of same order. Also, the condition $(F,A) \begin{matrix}
\subset \\
\approx
\end{matrix} (G,A)$ implies that $(F,A)\subseteq (G,A)$ and $(G,A)\subseteq (F,A)$. Now, using Theorem 5.5, we get, $\chi(e)_{(F,A)}$ = $\chi(e)_{(G,A)}$, for all $e\in A$. Here, we indicate  $\chi(e)_{(F,A)}$ to represent gravity of $e\in A$ in $(F,A) $ and $\chi(e)_{(G,A)}$ to indicate gravity of $e\in A$ in $(G,A)$. Again, since $X$ is ordered while representing the matrices and the ordering of the attributes is independent of choice, hence the 1's in each of the matrix will be placed in the same ordered place, which results the same matrices for both the soft sets. Thus, we get $Sim\{(F,A),(G,B)\}=1$.
\end{proof}

\begin{theorem}
Let $(F,A)$ and $(G,A)$ be two soft sets defined over a universal set $X$. If $(F,A) \begin{matrix}
\supset \\
\approx
\end{matrix} (G,A)$, then $Sim((F,A),(G,A))=1$,  in some specific ordering of the attributes in the matrix representation of the soft sets.
\end{theorem}
\begin{proof}
Similar to the above proof.
\end{proof}

\begin{theorem}
Let $(F,A)$ and $(G,A)$ be two soft sets defined over a universal set $X$. We consider $M$ and $N$ to be the matrices of  $(F,A)$ and $(G,A)$ respectively, where $|A|=n$. If the columns of $M$ and $N$ are linearly independent and span $R^n$, then $Sim\{(F,A),(G,A)\}=1$, in some specific ordering of the attributes in the matrix representation of the soft sets.
\end{theorem}
\begin{proof}
Since the columns of $M$ and $N$ are linearly independent and span $R^n$, hence they form a basis for $R^n$. Also, since the entries of $M$ and $N$ are $0$ and $1$ only, they will form the identical basis i.e., the set $\{(1,0,0,\dots,0 ), (0,1,0,\dots,0), \dots ,(0,0,0,\dots, 1)\}$. Again, the set of attributes are same for both the matrices, hence the matrices will be identical to each other if we arrange the columns in a specific order. Thus, obviously for two identical matrices,  we get $Sim\{(F,A),(G,A)\}=1$.
\end{proof}

\begin{corollary}
For two soft sets $(F,A)$ and $(G,A)$ defined over a universal set $X$, if $F$ and $G$ are one-one and every element of the families $\tau(F,A)$ and $\tau(G,A)$ belongs to $MIN(\tau(F,A))$ and $MIN(\tau(G,A))$ respectively, then $Sim\{(F,A),(G,A)\}=1$, in some specific ordering of the attributes in the matrix representation of the soft sets.
\end{corollary}

\begin{proof}
If every element of the families $\tau(F,A)$ and $\tau(G,A)$ belongs to $MIN(\tau(F,A))$ and $MIN(\tau(G,A))$ respectively, then in the matrix representation of $(F,A)$ and $(G,B)$; every column in each matrix will linearly independent, since all the elements of $\tau(F,A)$ and $\tau(G,B)$ are minimal and hence does not contain in any other element of $\tau(F,A)$ and $\tau(G,B)$ respectively. Again, the mappings $F$ and $G$ are one-one, hence, let $|\tau(F,A)|=|A|=|\tau(G,A)|=n$. Thus, the matrices of $(F,A)$ and $(G,B)$ will be of same order and the linearly independent columns will span $R^n$. Hence, by Theorem 5.9, $Sim\{(F,A),(G,A)\}=1$.
\end{proof}
\noindent
Let us consider an example to illustrate the above result.
\begin{example}
For two soft sets $(F, A)$ and $(G, A)$ defined over a universal set $X$,
let $X=\{a, b, c\}$,
$A$=$\{x, y, z\}$.
We define  $(F, A)$, where
$F(x)=\{b,c\}$, $F(y)=\{c,a\}$, $F(z)=\{a,b\}$.
Hence, $\tau(F,A)=\{\{a,b\},\{b,c\},\{c,a\}\}$
Similarly, we define $(G,A)$, where
$G(x)=\{c,a\}$, $G(y)=\{a,b\}$, $G(z)=\{b,c\}$.
Hence, $\tau(G,A)=\{\{a,b\},\{b,c\},\{c,a\}\}$
Here, the mappings $F$ and $G$ are one-one and all the elements of the families $\tau(F,A)$ and $\tau(G,A)$ belong to $MIN(\tau(F,A))$ and $MIN(\tau(G,A))$ respectively. The binary representations  of $(F,A)$ and $(G,A)$ are given by $M'$ and $N'$.

\begin{center}
   
$M=$
\begin{tabular}{c|c c c c c}

& F(x) & F(y) & F(z) \\
\hline
$a$ & 0 & 1 & 1 \\
$b$ & 1 & 0 & 1 \\
$c$ & 1 & 1 & 0 \\

\end{tabular},
$N=$
\begin{tabular}{c|c c c c c}

& G(z) & G(x) & G(y) \\
\hline
$a$ & 0 & 1 & 1 \\
$b$ & 1 & 0 & 1 \\
$c$ & 1 & 1 & 0 \\

\end{tabular}
\end{center}
Thus, we get following two matrices $M$ and $N$ for $(F,A)$ and $(G,A)$ respectively. 

\begin{center}
    
$M=
\begin{pmatrix}
0 & 1 & 1\\
1 & 0 & 1\\
1 & 1 & 0
\end{pmatrix}$,
$N=
\begin{pmatrix}
0 & 1 & 1  \\
1 & 0 & 1 \\
1 & 1 & 0 
\end{pmatrix}$

\end{center}
Thus, $Sim\{(F,A),(G,B)\}=1$.
\end{example}
\begin{corollary}
For two soft set $(F,A)$ and $(G,A)$ defined over a universal set $X$, if $F$ and $G$ are one-one and every element of the families $\tau(F,A)$ and $\tau(G,A)$ belongs to $MAX(\tau(F,A))$ and $MAX(\tau(G,A))$ respectively, then $Sim\{(F,A),(G,A)\}=1$, in some specific ordering of the attributes in the matrix representation of the soft sets.
\end{corollary}
\begin{proof}
Similar to the above proof.
\end{proof}

\section{Discussion}
The theory of soft set has been applied wrongly  to many areas of mathematics and allied areas by almost all the researchers of soft set theory except Molodtsov \cite{2, 3}.As he has stated multiple times \cite{2,3}, correctness of soft set theory is required to maintain the genuine philosophy of soft set theory's importance. Thus, we have tried to establish some results and notions based on the correct structures introduced by only Molodtsov \cite{1,2,3,4}. Similarity measure discussed here may be applied in granular computing as well as in other computing paradigm. Based on our present study, we propose the following conjecture. \\\\
\textbf{Conjecture 1:} If a relation $\Omega$ is correct for any quadruple of pairwise equivalent soft sets $(F,A)\cong (F',A')$
 and $(G,B)\cong (G',B')$, defined over a universe $X$, then it is not necessary that $Sim \{(F,A), (G,B)\}$ = $Sim \{(F',A'),(G',B')\}$.\\\\
 \noindent
 Not only we are trying to follow correct notions and philosophy of soft set theory, but also several pioneers of fuzzy set theory  corrected several misconceptions related to fuzzy set theory. Fuzzy pioneer George J. Klir and his co-authors \cite{22} pointed out several misconceptions available in literature of philosophy of concepts and fuzzy set theory. But in our case, not only a part of soft set theory is wrong but almost all the fundamental operations of soft sets are completely wrong along with the available notion of soft set theory by Çağman and Enginoğlu \cite{9}. For example, we can raise question on definition of soft topology \cite{23,24,25,26} because the definition of empty soft set provided by Molodtsov \cite{3} is completely different than that of the notion of empty set  available literature of soft set theory. In case of Molodtsov's definition of empty soft set, an empty set of parameters plays a crucial role, but this idea is absent in the definition of empty set provided by Maji et al. \cite{6}. Since Molodtsov, the father of soft set theory, raised concern about incorrect notions, operations and thus related results of soft set theory, hence it is now prime duty of the community of soft set researchers to look back to the beginning of soft set theory and hybrid structures by following the correct path of Molodtsov. Although there are thousands of published papers available on soft set theory and related areas , but it is not our intention to encourage wrong ideas related to soft set theory and hybrid structures. Thus in this paper, we do not focus to apply our results in different areas but our main intention is to develop correct theories of soft sets by following  Molodtsov \cite{1,2,3,4}.
 
\section{Conclusion}

In this paper, we  develop some results based on correct notions  of soft sets introduced by Molodtsov \cite{1,2,3,4}. Some unary and binary operations on soft sets are defined correctly in matrix forms. Also,  similarity measure between two soft sets $(F,A)$ and $(G,B)$ are defined  based on the families $\tau(F,A)$ and $\tau(G,B)$, but not on the set of parameters \cite{11,12}. We hope to continue the process of establishing various ideas related to correct notions of Molodtsov in forth coming papers. We also hope that this paper will attract the attention of scientific community related to soft set and hybrid structures; and it will considered one of important paper in the history of soft set theory because this paper reports for the first time the comment of Molodtsov on social platform regarding incorrectness of available definition of soft set . \\

\noindent
{\bf CRediT authorship contribution statement}\\
Santanu Acharjee: Conceptualization, Formal analysis, Investigation, Methodology, Writing - original draft, Writing- review $ \&$ editing.\\\\
\noindent
Amlanjyoti Oza: Conceptualization, Formal analysis, Investigation, Methodology, Writing - original draft, Writing- review $ \&$ editing.\\

\noindent
{\bf Declaration of Competing Interest}
The authors declare that they have no known competing financial interests or personal relationships that could have appeared to influence the work reported in this paper.

\end{document}